\documentclass[11pt,twoside]{article}

\usepackage{amsfonts}
\usepackage{amsmath}
\usepackage{amsthm}
\usepackage[mathscr]{euscript}
\usepackage{mathrsfs}
\usepackage{indentfirst}
\usepackage{booktabs,multirow,bigstrut,makecell}
\usepackage{color}\usepackage{enumitem}

\newtheorem{theorem}{Theorem}[section]
\newtheorem{lemma}[theorem]{Lemma}

\numberwithin{equation}{section}

\newcommand{\lbl}[1]{\label{#1}}
\allowdisplaybreaks

\newcommand{\be}{\begin{equation}}
\newcommand{\ee}{\end{equation}}
\newcommand\bes{\begin{eqnarray}} \newcommand\ees{\end{eqnarray}}
\newcommand{\bess}{\begin{eqnarray*}}
\newcommand{\eess}{\end{eqnarray*}}
\newcommand{\bbbb}{\left\{\begin{aligned}}
\newcommand{\nnnn}{\end{aligned}\right.}
\newcommand{\bea}{\begin{align*}}
\newcommand{\eea}{\end{align*}}

\newcommand\ep{\varepsilon}

\newcommand\kk{\left}
\newcommand\rr{\right}
\newcommand\dd{\displaystyle}

\newcommand\df{\dd\frac}

\newcommand\ap{\alpha}

\newcommand\yy{\infty}
\newcommand\qq{\eqref}

\newcommand\R{\mathbb{R}}
\newcommand\ol{\overline}

\newcommand\br{\bar R}

\begin{document}
\setlength{\baselineskip}{16pt} \pagestyle{myheadings}

\begin{center}{\Large\bf Free boundary problems with nonlocal}\\[2mm]
{\Large\bf  and local diffusions I: Global solution\footnote{This work was supported by NSFC Grant 11771110}}\\[4mm]
  {\Large  Jianping Wang, \ Mingxin Wang\footnote{Corresponding author. {\sl E-mail}: mxwang@hit.edu.cn}}\\[1.5mm]
{School of Mathematics, Harbin Institute of Technology, Harbin 150001, PR China}
\end{center}

\date{\today}

\begin{abstract}We study a class of free boundary problems of ecological models with nonlocal and local diffusions, which are natural extensions of free boundary problems of reaction diffusion systems in there local diffusions are used to describe the population dispersal, with the free boundary representing the spreading front of the species. We prove that such kind of nonlocal and local diffusion problems has a unique global solution, and then show that a spreading-vanishing dichotomy holds. Moreover, criteria of spreading and vanishing, and long time behavior of solution when spreading happens are established for the classical Lotka-Volterra competition and prey-predator models. Compared with free boundary systems with local diffusions \cite{GW12, WZjdde14, WZjdde17} as well as with nonlocal diffusions \cite{DWZ19}, the present paper involves some new difficulties, which should be overcome by use of new techniques. This is part I of a two part series, where we prove the existence,  uniqueness, regularity and estimates of global solution. The spreading-vanishing dichotomy, criteria of spreading and vanishing, and long-time behavior of solution when spreading happens will be studied in the separate part II.



\end{abstract}

\section{Introduction}
\setcounter{equation}{0} {\setlength\arraycolsep{2pt}
\markboth{J. P. Wang and M. X. Wang}{Free boundary problems with nonlocal and local diffusions}

The spreading and vanishing of multiple species is an important content in understanding ecological complexity. In order to study the spreading and vanishing phenomenon, many mathematical models have been established. The logistic equation, competition and prey-predator models with local diffusions and free boundaries have been studied widely by many authors, please refer to, for example, \cite{1-DLin10} for the logistic equation, \cite{GW12, WZjdde14}, \cite{DLdcds14}-\cite{7-WZna17} for the competition models, \cite{ZW2014}-\cite{ZWaa15} for the prey-predator models, and the references therein. The general form of single equation with local diffusion and free boundaries is (\cite{2-W18}):
 \bes
 \left\{\begin{array}{lll}
 u_t-du_{xx}=f(t,x,u), \ \ &t>0,\, \ g(t)<x<h(t),\\[.5mm]
 u(t,x)=0, \ \ &t>0,\, \ x=g(t),\ h(t),\\[.5mm]
 g'(t)=-\mu u_x(t,g(t)),\ \ \ \ &t>0,\\[.5mm]
 h'(t)=-\mu u_x(t,h(t)),\ \ \ \ &t>0,\\[.5mm]
 u(0,x)=u_0(x),&|x|\le h_0,\\[.5mm]
 h(0)=-g(0)=h_0.
 \end{array}\right.
 \label{1.1}\ees
For the deduction of the free boundary condition $h'(t)=-\mu u_x(t,h(t))$, please refer to \cite{10-hdu12}.

It is well known that random dispersal or local diffusion describes the movements of organisms between adjacent spatial locations. It has been increasingly recognized the movements and interactions of some organisms can occur between non-adjacent spatial locations. The evolution of nonlocal diffusion has attracted a lot of attentions for both theoretically and empirically; see \cite{Nathan12}-\cite{8-BJMB16} and references therein. An extensively used nonlocal diffusion operator to replace the local diffusion term $d\Delta u$ (the Laplacian operator in $\mathbb{R}^N$) is given by
 \[d(J * u-u)(t,x):=d\left(\int_{\mathbb R^N} J(x-y)u(t,y)dy-u(t,x)\right).\]

To describe the spatial spreading of species in the nonlocal diffusion processes, recently, the authors of \cite{CDLL} studied the following free boundary problem of Fisher-KPP  nonlocal diffusion model:
 \begin{equation}\left\{\begin{aligned}
&u_t=d\int_{g(t)}^{h(t)}\!J(x-y)u(t,y){\rm d}y-du(t,x)+f(t,x,u), & &t>0,~g(t)<x<h(t),\\
&u(t,g(t))=u(t,h(t))=0,& &t>0,\\
&h'(t)=\mu\int_{g(t)}^{h(t)}\!\!\int_{h(t)}^{\infty}
J(x-y)u(t,x){\rm d}y{\rm d}x,& &t>0,\\
&g'(t)=-\mu\int_{g(t)}^{h(t)}\!\!\int_{-\infty}^{g(t)}
J(x-y)u(t,x){\rm d}y{\rm d}x,& &t>0,\\
&u(0,x)=u_0(x),~h(0)=-g(0)=h_0,& &|x|\le h_0,
 \end{aligned}\right.
 \label{1.2}\end{equation}
where $x=g(t)$ and $x=h(t)$ are free boundaries
to be determined together with $u(t,x)$, which is always assumed to be identically $0$ for $x\in \mathbb{R}\setminus [g(t), h(t)]$; $d$, $\mu$ and $h_0$ are positive constants. The kernel function $J: \mathbb{R}\rightarrow\mathbb{R}$ is continuous and satisfies
 \begin{enumerate}[leftmargin=3em]
\item[{\bf(J)}] $J(0)>0,~J(x)\ge 0, \ \dd\int_{\mathbb{R}}J(x){\rm d}x=1$, \ $J$\, is\, symmetric, \ and\ $\dd\sup_{\mathbb{R}}J<\infty$.
 \end{enumerate}
The reaction function $f(t,x,u)$ has logistic structure. It was shown in \cite{CDLL} that the problem \eqref{1.2} has a unique global solution. Furthermore, the spreading-vanishing dichotomy about free boundary problems of local diffusive logistic equation (\cite{1-DLin10}) holds true for the nonlocal diffusive problem \eqref{1.2} when $f(t,x,u)=f(u)$. However, from \cite[Remark 1.4]{CDLL} we know that when $d\le f'(0)$, spreading happens no matter how small $h_0,\mu$ and $u_0$ are. This is very different from the spreading-vanishing
criteria for the local diffusion models.

Motivated by the papers \cite{CDLL} and \cite{GW12, WZjdde14, WZhang16, ZhaoW16, WWmma18} (two species local diffusion systems with common free boundary), the authors of \cite{DWZ19} studied the following free boundary problem of nonlocal diffusive system
 \bes
\left\{\begin{aligned}
&u_{it}=d_i\dd\int_{g(t)}^{h(t)}\!J_i(x-y)u_i(t,y){\rm d}y-d_iu_i+f_i(t, x,u_1,u_2), &&t>0,~g(t)<x<h(t),\\
&u_i(t,g(t))=u_i(t,h(t))=0, &&t\ge 0,\\
&h'(t)=\dd\sum_{i=1}^2\mu_i \int_{g(t)}^{h(t)}\!\!\int_{h(t)}^{\infty}
J_i(x-y)u_i(t,x){\rm d}y{\rm d}x, &&t\ge 0,\\
&g'(t)=-\dd\sum_{i=1}^2\mu_i\int_{g(t)}^{h(t)}\!\!\int_{-\infty}^{g(t)}\!
J_i(x-y)u_i(t,x){\rm d}y{\rm d}x,\ &&t\ge 0,\\
&u_i(0,x)=u_{i0}(x),\ \ \ h(0)=-g(0)=h_0, &&|x|\le h_0,\\
&i=1,\,2.
\end{aligned}\right.\label{1.2a}
 \ees
They proved the existence and uniqueness of global solution, a spreading-vanishing dichotomy and obtained the criteria for spreading and vanishing.

Kao et al. \cite{6-KLS10} studied the competition model in which one diffusion is local and the other one is nonlocal:
  $$\left\{\begin{aligned}
& u_t=d_1\Delta u+ u(a-u-v), && t>0,~x\in\Omega,\\
&v_t=d_2\dd\int_{\Omega}J(x-y)v(t,y){\rm d}y-d_2v+v(a-u-v),
&& t>0,~x\in\Omega.
\end{aligned}\right.$$

Motivated by the above mentioned works, in this paper we discuss some ecological models with nonlocal and local diffusions and free boundaries. Based on the deductions of free boundary conditions in \eqref{1.1} and \eqref{1.2}, it is reasonable to study the following free boundary problems:
 \bes
\left\{\begin{aligned}
&u_t=d_1\int_{g(t)}^{h(t)}\!J(x,y)u(t,y){\rm d}y-d_1u+f_1(t,x,u,v), &&t>0,~g(t)<x<h(t),\\
&v_t=d_2 v_{xx}+f_2(t,x,u,v), &&t>0,~g(t)<x<h(t),\\
&u(t,g(t))=u(t,h(t))=v(t,g(t))=v(t,h(t))=0, &&t\ge 0,\\
&h'(t)=-\mu v_x(t,h(t))+\rho\int_{g(t)}^{h(t)}\!\!\int_{h(t)}^\infty\! J(x,y)u(t,x){\rm d}y{\rm d}x, &&t\ge 0,\\[1mm]
&g'(t)=-\mu v_x(t,g(t))-\rho\int_{g(t)}^{h(t)}\!\!\int_{-\infty}^{g(t)}\! J(x,y)u(t,x){\rm d}y{\rm d}x, &&t\ge 0,\\
&u(0,x)=u_0(x), \ v(0,x)=v_0(x),\ h(0)=-g(0)=h_0>0, &&|x|\le h_0,
\end{aligned}\right.
 \label{1.3}
 \ees
where $J(x,y)=J(x-y)$; $[-h_0,h_0]$ represents the initial population range of the species $u$ and $v$; $x=g(t)$ and $x=h(t)$ are the free boundaries to be determined together with $u(t,x)$ and $v(t,x)$, which are always assumed to be identically $0$ for $x\in \mathbb{R}\setminus [g(t), h(t)]$; $d_i$ and $\mu,\rho$ are positive constants.

Denote by $C^{1-}(\Omega)$ the Lipschitz continuous function space in $\Omega$. We assume that the initial functions $u_0,v_0$ satisfy
 \begin{equation}
(u_0,v_0)\in C^{1-}([-h_0,h_0])\times W^2_p(-h_0,h_0), \ u_0(\pm h_0)=v_0(\pm h_0)=0, \ u_0,\,v_0>0~\text{ in }~(-h_0,h_0)
\label{1.4}
 \end{equation}
with $p>3$. The kernel function $J$ is supposed to satisfy
 \begin{enumerate}[leftmargin=4em]
\item[{\bf(J1)}] The condition {\bf(J)} holds, and $J\in C^{1-}(\R)$.
 \end{enumerate}
It follows from {\bf(J)} that there exist constants $\bar\varepsilon\in(0,h_0/4)$ and $\delta_0>0$ such that
 \bes
J(x,y)>\delta_0\ \ \ {\rm if}\ \ |x-y|<\bar\varepsilon.
 \label{1.5}\ees
The growth terms $f_i: \mathbb{R}^+\times\mathbb{R}\times\mathbb{R}^+\times\mathbb{R}^+
\rightarrow\mathbb{R}$ are assumed to be continuous and satisfy
 \vspace{-2mm}\begin{enumerate}[leftmargin=3em]
\item[{\bf(f)}] $f_1(t,x,0,v)=f_2(t,x,u,0)=0$, $f_i(t,x,u, v)$
is differentiable with respect to $u,v\in\mathbb{R}^+$, and for any $c_1, c_2>0$, there exists a constant $L(c_1, c_2)>0$ such that
 \[|f_i(t,x,u_1,v_1)-f_i(t,x,u_2,v_2)|\le L(c_1, c_2)(|u_1-u_2|+|v_1-v_2|),\ i=1,2\]
for all $u_1, u_2\in [0, c_1]$, $v_1, v_2\in [0, c_2]$ and all $(t,x)\in \mathbb{R}^+\times \mathbb{R}$;\vspace{-2mm}
\item[{\bf(f1)}] There exist $k_0>0$ and $r>0$ such that for all $v\ge 0$ and $(t,x)\in \mathbb{R}^+\times \mathbb{R}$, there hold: $f_1(t,x,u,v)<0$ when $u>k_0$, $f_1(t,x,u,v)\le ru$ when $0<u\le k_0$;\vspace{-2mm}
\item[{\bf(f2)}] For the given $k>0$, there exists $\Theta(k)>0$ such that $f_2(t,x,u,v)<0$ for $0\le u\le k$, $v\ge\Theta(k)$ and $(t,x)\in\mathbb{R}^+\times\mathbb{R}$;\vspace{-2mm}
\item[{\bf(f3)}] $f_{ix}(t,x,u,v)$ is continuous and for any $c_1, c_2>0$, there exists a constant $L^*(c_1, c_2)>0$ such that
 \[|f_i(t,x,u,v)-f_i(t,y,u,v)|\le L^*(c_1, c_2)|x-y|, \ \ i=1,\,2\]
for all $u\in [0, c_1]$, $v\in [0, c_2]$ and all $(t,x,y)\in \mathbb{R}^+\times\mathbb{R}\times\mathbb{R}$.
 \vspace{-2mm}\end{enumerate}

The condition {\bf(f)} implies
 \[|f_1(t,x,u,v)|\le L(c_1, c_2)u, \ \ |f_2(t,x,u,v)|\le L(c_1, c_2)v\]
 for all $u\in[0, c_1]$, $v\in[0, c_2]$ and all $(t,x)\in \mathbb{R}^+\times \mathbb{R}$.

Except where otherwise stated, we always assume that {\bf (f)}-{\bf (f3)} hold, the kernel
function $J$ satisfies {\rm \bf (J1)} and $u_0$, $v_0$ satisfy the condition \eqref{1.4} throughout this paper. We write $\|\phi,\varphi\|\le M$ means that $\|\phi\|\le M$, $\|\varphi\|\le M$.

Since this paper is very long, and the techniques used in the first part are rather different from those in the second part, it is divided into two separate parts. Part I here is mainly concerned with the existence, uniqueness, regularity and estimates of global solution. Part II focuses on the spreading-vanishing dichotomy, criteria of spreading and vanishing, and long time behavior of solution when spreading happens.

\section{Existence, uniqueness, regularity and estimates of global solution of (\ref{1.3})}
\setcounter{equation}{0} {\setlength\arraycolsep{2pt}

For convenience, we first introduce some notations. Let $L(u_0)$ and $L(J)$ be the Lipschitz constants of $u_0$ and $J$, respectively.
Let $k_0,\Theta(\cdot)$ be given in {\bf(f1)}, {\bf(f2)}. Denote
 \bess
 &k_1=\max\left\{\|u_0\|_\infty,\,k_0\right\},\ \ k_2=\max\left\{\|v_0\|_\infty,\,\Theta(k_1)\right\},\ \ L=L(k_1,k_2),&\\[1mm]
 &L^*=L^*(k_1,k_2),\ \ \ k_3=\dd\max\left\{\frac 1{h_0},\ \ \sqrt{\frac{L}{2d_2}}, \
 \frac {\|v_0'\|_{C([-h_0,h_0])}}{k_2}\right\},&\\[1mm]
 &x(t,y)=\dd\frac{(h(t)-g(t))y+h(t)+g(t)}2,\ \ \ y(t,x)=\frac{2x-g(t)-h(t)}{h(t)-g(t)},&\\[1mm]
 &\xi(t)=\df{4}{(h(t)-g(t))^2},\ \ \ \zeta(t,y)=\df{h'(t)+g'(t)}{h(t)-g(t)}+\df{(h'(t)-g'(t))y}{h(t)-g(t)},&\\[1mm]
  &\Sigma=[-1,1],\ \ \Pi_s=[0,s]\times\Sigma, \ \ R(t)=\mu k_3+2(h_0\rho k_1+\mu k_3)e^{\rho k_1t}.&
 \eess
For the given $T>0$, define
 \bess
\mathbb H^T&=&\left\{h\in C^1([0,T])~:~h(0)=h_0,
\; 0<h'(t)\le R(t)\right\},\\[1mm]
\mathbb G^T&=&\left\{g\in C^1([0,T])~:-g\in\mathbb{H}^T\right\}.
 \eess
And for $g\in \mathbb G^T$, $h\in\mathbb H^T$, define
 \bess
 D^T_{g,h}&=&\left\{(t,x)\in\mathbb{R}^2:\, 0<t\leq T,~g(t)<x<h(t)\right\},\\[1mm]
 \mathbb X_1^T&=&\mathbb X^T_{u_0,g,h}=\big\{\varphi\in C(\overline D^T_{g,h}):~0\le\varphi\le k_1,\ \varphi\big|_{t=0}=u_0,\; \varphi\big|_{x=g(t),h(t)}=0\big\},\\[1mm]
 \mathbb X_2^T&=&\mathbb X^T_{v_0,g,h}=\big\{\varphi\in C(\overline D^T_{g,h}):~0\le\varphi\le k_2,\ \varphi\big|_{t=0}=v_0,\; \varphi\big|_{x=g(t), h(t)}=0\big\},
 \eess
as well as
\[\mathbb X_{g,h}^T:=\mathbb X_1^T\times\mathbb X_2^T.\]

The following theorem is our main result in this part.

\begin{theorem}\label{th2.1}\,The problem \eqref{1.3} has a unique local solution $(u,v,g,h)$ defined on $[0,T]$ for some $0<T<\yy$. Moreover, $(g, h)\in\mathbb G^T\times\mathbb H^T$, $(u,v)\in\mathbb X^T_{g,h}$  and
 \bes
 &u\in C^{1,1-}(\ol D^T_{g,h}),\ \ \ \ v\in W^{1,2}_p(D^T_{g,h}),&
 \lbl{1.6}\\[1mm]
 &0< u\le k_1,\ \ 0<v\le k_2\ \ \ {\rm in}\ \ D^T_{g,h},& \label{2.1}\\[.5mm]
 &0<-v_x(t,h(t)),\ v_x(t,g(t))\le k_3, \ \ 0<t\le T,&
 \label{2.2}\ees
where $u\in C^{1,1-}(\ol D^T_{g,h})$
means that $u$ is differentiable continuously in $t\in[0,T]$ and is Lipschitz continuous in
$x\in[g(t),h(t)]$.

If we further assume that
\vspace{-2mm}\begin{enumerate}[leftmargin=4em]
 \item[{\bf(f4)}] For any given
$\tau$, $l$, $c_1$, $c_2>0$, there exists a constant $\bar L(\tau, l, c_1, c_2)$
such that
 \bes\|f_2(\cdot,x,u,v)\|_{C^{\frac\ap 2}([0,\tau])}\le \bar L(\tau, l, c_1, c_2)
 \lbl{1.8}\ees
for all $ x\in[-l,l], \ u\in[0,c_1], \ v\in[0,c_2]$.
 \vspace{-2mm}\end{enumerate}
Then the solution $(u,v, g, h)$ exists globally. Moreover, for any given $\tau>0$, \qq{2.1} and \qq{2.2} hold with $T$ replaced by $\tau$, and
 \bes
 g,\,h\in C^{1+\alpha/2}([0,\tau]), \ \  u\in C^{1,1-}(\ol D^\tau_{g,h}), \ \ v\in C^{1+\alpha/2,\,2+\alpha}((0,\tau]\times[g(t),h(t)]).
 \lbl{b.3}\ees
\end{theorem}

For the classical competition and prey-predator models
 \bes
  &&\hspace{-1.3cm}\mbox{\it Competition\, Model}: \ \ \ f_1=u(a-u-bv), \ \ f_2=v(1-v-cu),
  \lbl{1.9}\\[1mm]
 &&\hspace{-1.3cm}\mbox{\it Prey-predator\, Model}: \ f_1=u(a-u-bv), \ \ f_2=v(1-v+cu),
  \lbl{1.10}\ees
the conditions {\bf(f)}--{\bf(f4)} hold, where $a,b,c$ are positive constants.

Due to the presence of the nonlocal diffusion and local diffusion, the methods that solve the local diffusion models are not applicable any more and the arguments for the nonlocal system developed in \cite{DWZ19, CDLL} are far from sufficient for the present stage, the proofs of Theorem \ref{th2.1} are highly non trivial. Our approach to prove Theorem \ref{th2.1} is based on the fixed point theorem. Some new ideas and delicate calculations are given in the proof of Theorem \ref{th2.1}.

The proof of Theorem \ref{th2.1} will be divided into several lemmas because it is too long. Throughout this paper we use $C$, $C'$, $C_i$ and $C_i'$ to represent general constants, which may not be the same in different places.

We first state the following {\it Maximum Principle} which will be used frequently in our analysis.

\begin{lemma}[Maximum Principle {\cite[Lemma 2.2]{CDLL}}]\label{l2.2}
Assume that $J$ satisfies {\bf(J)} and $d$ is a positive constant, and $(r, \eta)\in\mathbb G^T\times\mathbb H^T$. Suppose that $\psi, \psi_t\in C(\overline D^T_{\eta,r})$ and fulfill, for some $\varrho\in L^\infty (D^T_{\eta,r})$,
 \bess\left\{\begin{aligned}
&\psi_t\ge d\int_{\eta(t)}^{r(t)}\!J(x,y)\psi(t,y){\rm d}y-d\psi+\varrho\psi, && (t,x)\in D^T_{\eta,r},\\
& \psi(t, \eta(t)) \geq 0,\ \psi(t, r(t)) \geq 0, && 0\le t\le T,\\
&\psi(0,x)\ge0,  && |x|\le h_0.
 \end{aligned}\right.\eess
Then $\psi\ge0$ on $\overline D^T_{\eta, r}$. Moreover, if $\psi(0,x)\not\equiv0$ in $[-h_0, h_0]$, then $\psi>0$ in $D^T_{\eta,r}$.
\end{lemma}

\begin{lemma}\label{l2.3}
For any $T>0$ and $(g,h)\in \mathbb G^T\times\mathbb H^T$, the problem
 \bes
\left\{\begin{aligned}
&u_t=d_1\int_{g(t)}^{h(t)}\!J(x,y)u(t,y){\rm d}y-d_1u+f_1(t,x,u,v), &&(t,x)\in D^T_{g,h},\\
&v_t=d_2v_{xx}+f_2(t,x,u,v), &&(t,x)\in D^T_{g,h},\\
&u(t,g(t))=u(t,h(t))=v(t,g(t))=v(t,h(t))=0, &&0\le t\le T,\\
&u(0,x)=u_0(x),\ \ v(0,x)=v_0(x), &&|x|\le h_0
\end{aligned}\right.
 \label{2.3}
 \ees
admits a unique solution $(u_{g,h},v_{g,h})\in \mathbb X^T_{g,h}$, and
$(u_{g,h}, v_{g,h})$ satisfies \qq{2.1} and \qq{2.2}.
Moreover, $v_{g,h}\in W^{1,2}_p(D^T_{g,h})$.
\end{lemma}

\begin{proof}\,{\it Step 1}:\, For $\tilde u\in\mathbb X_1^s$ with $0<s\le T$, consider the following initial-boundary value problem
\bes
\left\{\begin{aligned}
&v_t=d_2v_{xx}+f_2(t,x,\tilde u,v), \ \ &&(t,x)\in D^s_{g,h},\\
&v(t,g(t))=v(t,h(t))=0, &&0\le t\le s,\\
&v(0,x)=v_0(x), &&|x|\le h_0.
\end{aligned}\right.
 \label{2.4}
 \ees
Let $z(t,y)=v(t,x(t,y))$, $\tilde w(t,y)=\tilde u(t,x(t,y))$. It follows from \eqref{2.4} that
 \bes
\left\{\begin{aligned}
&z_t=d_2\xi(t)z_{yy}+\zeta(t,y)z_y+f^*_2(t,y,\tilde w,z), &&0<t\le s,~|y|<1,\\
&z(t,\pm 1)=0, &&0\le t\le s,\\
&z(0,y)=v_0(h_0y)=:z_0(y), &&|y|\le 1,
\end{aligned}\right.
 \label{2.5} \ees
where $f^*_2(t,y,\tilde w,z)=f_2(t,x(t,y),\tilde w,z)$.
Note that $(g,h)\in\mathbb G_{h_0,s}\times\mathbb H_{h_0,s}$, we have $\xi\in C([0,s]),\,\zeta\in C(\Pi_s)$ and
  $$\|\xi\|_{L^\infty((0,s))}\le 1/h_0^2,\ \ \
  \|\zeta\|_{L^\infty(\Pi_s)}\le 2R(s)/h_0\le 2R(T)/h_0.$$
It is easy to see that $\tilde w\in C(\Pi_s)$ and $0\le\tilde w\le k_1$. Notice that $z_0(y)\in\overset{\circ}{W}{}^1_2(\Sigma)$. By the upper and lower solutions method and $L^2$ theory (\cite[Ch.\,III, Theorem 6.1]{L-S-Y1968}) we can show that the problem \eqref{2.5} has a unique solution $z\in W^{1,2}_2(\Pi_s)$, and $z\in C^{\alpha/2,\alpha}(\Pi_s)$ by the embedding theorem. Moreover, $0\le z\le k_2$ in $\Pi_s$ by the weak maximum principle. Hence, the problem \eqref{2.4} admits a unique solution $v\in \mathbb X_2^s$.

\vskip 4pt
{\it Step 2}:\, For $0<s\le T$, let $v$ be the unique solution of \eqref{2.4} and  consider
 \bes
\left\{\begin{aligned}
&u_t=d_1\int_{g(t)}^{h(t)}\!J(x,y) u(t,y){\rm d}y-d_1u+f_1(t,x,u,v), &&(t,x)\in D^s_{g,h},\\
&u(t,g(t))=u(t,h(t))=0, &&0\le t\le s,\\
&u(0,x)=u_0(x), &&|x|\le h_0.
\end{aligned}\right.
 \lbl{2.6} \ees
Thanks to \cite[Lemma 2.3]{CDLL}, this problem admits a unique solution $u$ which satisfies $0<u\le k_1$ for $(t,x)\in[0,s]\times(g(t),h(t))$. It is easily seen that $u\in\mathbb X_1^s$. Define a mapping $\mathcal{F}_s:\mathbb X_1^s\rightarrow\mathbb X_1^s$ by
\[\mathcal{F}_s \tilde u=u.\]
If $\mathcal{F}_s \tilde u=\tilde u$, then $(\tilde u,v)$ solves \eqref{2.3} in $D^s_{g,h}$.

\vskip 4pt{\it Step 3}:\, We shall prove that $\mathcal{F}_s$ has a fixed point in $\mathbb X_1^s$ provided $s$ small enough. Evidently, $\mathbb X_1^s$ is a closed bounded subset of $C(\ol D^s_{g,h})$. Let $\tilde u_1,\tilde u_2\in\mathbb X_1^s$ and $u_i=\mathcal{F}_s \tilde u_i$ with $i=1,2$. Let $v_i$ be the unique solution of \eqref{2.4} with $\tilde u_i$. Then $(u_i,v_i)\in\mathbb X^s$. Notice that $u_i$ satisfies
\bess
\left\{\begin{aligned}
&u_{i,t}=d_1\int_{g(t)}^{h(t)}\!J(x,y)u_i(t,y){\rm d}y-d_1u_i+f_1(t,x,u_i,v_i), &&t_x<t\le s,~g(t)<x<h(t),\\
&u_i(t_x,x)=\tilde u_0(x), &&g(s)<x<h(s),
\end{aligned}\right.
 \eess
where
 \begin{align*}
\tilde u_0(x)=\left\{\begin{aligned}
&0,& &|x|>h_0,\\
&u_0(x),& &|x|\le h_0,
\end{aligned}\right. \quad t_x=\left\{\begin{aligned}
&t_{x,g}& &\mbox{if }\ x\in[g(s),-h_0), \ x=g(t_{x,g}),\\
&0& &\mbox{if } \ |x|\le h_0,\\
&t_{x,h}& &\mbox{if } \ x\in(h_0,h(s)], \ x=h(t_{x,h}),
\end{aligned}
\right.
  \end{align*}
Let $\tilde u=\tilde u_1-\tilde u_2$, $u=u_1-u_2$ and $v=v_1-v_2$, we have
  \bes
\left\{\begin{aligned}
&u_t+a(t,x)u=d_1\int_{g(t)}^{h(t)}\!J(x,y)u(t,y){\rm d}y+b(t,x)v, &&t_x<t\le s,~g(t)<x<h(t),\\
&u(t_x,x)=0, &&g(s)<|x|<h(s),
\end{aligned}\right.
  \label{2.7}\ees
where
\bess
 a(t,x)&=&d_1-\int_0^1f_{1,u}(t,x,u_2+(u_1-u_2)\tau,v_2){\rm d}\tau,\\
 b(t,x)&=&\int_0^1f_{1,v}(t,x,u_1,v_2+(v_1-v_2)\tau){\rm d}\tau.
\eess
Recall {\bf(f)}, there holds that $\|a, b\|_\infty\le d_1+L=:L_1$. It follows from \eqref{2.7} that, for $x\in(g(t),h(t))$ and $t_x<t\le s$,
\bess
u(t,x)=e^{-\int_{t_x}^t a(\tau,x){\rm d}\tau}\int_{t_x}^t e^{\int_{t_x}^l a(\tau,x){\rm d}\tau}\left(d_1\int_{g(l)}^{h(l)}J(x,y)u(l,y){\rm d}y+b(l,x)v(l,x)\right){\rm d}l.
\eess
Due to $(g(t),h(t))\subset(g(s),h(s))$ for $t_x<t\le s$, this implies that
  \bes
|u(t,x)|\le e^{2L_1s}\left(d_1\|u\|_{C(\ol D^s_{g,h})}s+L_1\int_{t_x}^t |v(l,x)|{\rm d}l\right).
 \label{2.8}\ees
Note that $v$ satisfies
\bess
\left\{\begin{aligned}
&v_t=d_2v_{xx}+a_0(t,x)v+b_0(t,x)\tilde u, &&(t,x)\in D^s_{g,h},\\
&v(t,g(t))=v(t,h(t))=0, &&0\le t\le s,\\
&v(0,x)=0, &&|x|\le h_0,
\end{aligned}\right.
 \eess
where
\bess
a_0(t,x)&=&\int_0^1f_{2,v}(t,x,\tilde u_1,v_2+(v_1-v_2)\tau){\rm d}\tau,\\[.1mm]
b_0(t,x)&=&\int_0^1f_{2,u}(t,x,\tilde u_2+(\tilde u_1-\tilde u_2)\tau,v_2){\rm d}\tau.
\eess
Clearly, $\|a_0,b_0\|_\infty\le L$. Let
 \bess
 \tilde w(t,y)=\tilde u(t,x(t,y)),\ \ \tilde z(t,y)=v(t, x(t,y)),\ \
 \tilde a_0(t,y)=a_0(t,x(t,y)),\ \ \tilde b_0(t,y)=b_0(t,x(t,y)).
 \eess
It is easy to see that $\tilde z$ satisfies
\bess
\left\{\begin{aligned}
&\tilde z_t=d_2\xi(t)\tilde z_{yy}+\zeta(t,y)\tilde z_y
+\tilde a_0\tilde z+\tilde b_0\tilde w, &&0<t\le s,~|y|<1,\\
&\tilde z(t,\pm 1)=0, &&0\le t\le s,\\
&\tilde z(0,y)=0, &&|y|\le 1.
\end{aligned}\right.
 \eess
Thanks to the parabolic $L^p$ theory, one can obtain that, with $p>3$ and $\alpha=1-3/p$,
 \[\|\tilde z\|_{W^{1,2}_p(\Pi_s)}\le C\|\tilde w\|_{C(\Pi_s)}
  =C\|\tilde u\|_{C(\ol D^s_{g,h})}.\]
Using the arguments in the proof of \cite[Theorem 1.1]{2-W18} we have
 \bes
[\tilde z,\,\tilde z_y]_{C^{\alpha/2,\alpha}(\Pi_s)}\le C'\|\tilde z\|_{W^{1,2}_p(\Pi_s)}\le C'C\|\tilde u\|_{C(\ol D^s_{g,h})},
 \lbl{2.9}\ees
where $C'$ is independent of $s^{-1}$, and $[\,\cdot\,]_{C^{\frac\alpha 2,\alpha}(\Pi_s)}$ is the H\"{o}lder semi-norm. It follows from $\tilde z(0,y)=0$ that $\|\tilde z\|_{L^\infty(\Pi_s)}\le [\tilde z]_{C^{\alpha/2,\alpha}(\Pi_s)}s^{\alpha/2}$. Thus we have, for $t_x\le t\le s\le 1$,
 \bess
\int_{t_x}^t |v(l,x)|{\rm d}l\le \int_0^s\|\tilde z\|_{L^\infty(\Pi_s)}{\rm d}l\le s[\tilde z]_{C^{\alpha/2,\alpha}(\Pi_s)}
\le sC'C\|\tilde u\|_{C(\ol D^s_{g,h})}.
 \eess
Inserting this into \eqref{2.8} gives
  \[|u(t,x)|\le e^{2L_1s}\big(d_1s\|u\|_{C(\ol D^s_{g,h})}+L_1 C'C s\|\tilde u\|_{C(\ol D^s_{g,h})}\big).\]
Taking $s$ small enough such that
 \[d_1s e^{2L_1s}\le 1/2,\ \ \ L_1C'C s e^{2L_1s}\le 1/4.\]
Then $\|u\|_{C(\ol D^s_{g,h})}\le \frac12\|\tilde u\|_{C(\ol D^s_{g,h})}$.
The contraction mapping theorem shows that $\mathcal{F}_s$ has a unique fixed point $u$ in $\mathbb X_1^s$. Let $z$ be the unique solution of \eqref{2.5} with $\tilde w(t,y)$ replaced by $w(t,y)=u(t,x(t,y))$.

\vskip 4pt{\it Step 4}:\, The local existence and uniqueness of solution $(u,v)$ of \eqref{2.3}. From the above analysis, the function $v(t,x)=z(t,y(t,x))$ solves \qq{2.4} with $\tilde u$ replaced by $u$ and $v\in \mathbb X_2^s$. Hence, $(u,v)\in\mathbb X_{g,h}^s$ solves \eqref{2.3} with $T$ replaced by $s$. Moreover, from the above arguments we know that any solution $(U,V)$ of \eqref{2.3} in $(0, s]$ satisfies $(U,V)\in \mathbb X^s_{g,h}$. Hence, $(u,v)$ is the unique solution of \eqref{2.3} in $(0, s]$.

\vskip 4pt{\it Step 5}:\, We finally show that the unique solution $(u,v)$ of \eqref{2.3} can be extended to $D^T_{g,h}$. It is clear that $u(s,x)\in C([g(s),h(s)])$, $0\le u(s,x)\le k_1$, $0\le v(s,x)\le k_2$ and
 $$u(s,g(s))=u(s,h(s))=v(s,g(s))=v(s,g(s))=0.$$

Same as the above, let $z(t,y)=v(t,x(t,y))$, $w(t,y)=u(t,x(t,y))$. Since $z_0(y)=v_0(h_0y)\in W^2_p(\Sigma)$ and $p>3$, where $\Sigma=[-1,1]$, applying the $L^p$ theory to \qq{2.5} and the uniqueness of weak solution, we have $z\in W^{1,2}_p(\Pi_s)\hookrightarrow C^{(1+\alpha)/2,1+\alpha}(\Pi_s)$. And so $z(s,\cdot)\in \overset{\circ}{W}{}^1_2(\Sigma)$.
Note that in the above Steps $1,2,3$ we only used $u_0\in C([-h_0,h_0])$, $z_0(y)\in\overset{\circ}{W}{}^1_2(\Sigma)$ without using $u_0\in C^{1-} ([-h_0,h_0])$ and  $z_0\in W^2_p(\Sigma)$. We can apply the above Steps $1,2,3$ to \eqref{2.3} but with initial time $t=0$ replaced by $t=s$ to get an $\bar s>s$ and a unique $(\hat u,\hat z)$ which satisfies
 \bess
\left\{\begin{aligned}
&\hat u_t=d_1\int_{g(t)}^{h(t)}\!J(x,y)\hat u(t,y){\rm d}y-d_1\hat u+f_1(t,x,\hat u,\hat v), &&s<t\le\bar s, \ g(t)<x<h(t),\\
&\hat u(t,g(t))=\hat u(t,h(t))=0, &&s\le t\le\bar s,\\
&\hat u(s,x)=u(s,x), &&g(s)\le x\le h(s),
\end{aligned}\right.
  \eess
and
 \bess
\left\{\begin{aligned}
&\hat z_t=d_2\xi(t)\hat z_{yy}+\zeta(t,y)\hat z_y+f^*_2(t,y,\hat w,\hat z), &&s<t\le\bar s,~|y|<1,\\
&\hat z(t,\pm 1)=0, &&s\le t\le\bar s,\\
&\hat z(s,y)=v(s, x(s,y)), &&|y|\le 1
\end{aligned}\right.
  \eess
as well as $\hat u,\hat v\in C([s,\bar s]\times[g(t),h(t)])$, where $\hat v(t,x(t,y))=\hat z(t,y)$, $\hat w(t,y)=\hat u(t,x(t,y))$. Set $u(t,x)=\hat u(t,x), z(t,y)=\hat z(t,y)$ for $t\in[s,\bar s], g(t)\le x\le h(t)$, $|y|\le 1$. Clearly, $u\in C(\ol D^{\bar s}_{g,h})$ solves \qq{2.6} with $(s, v)$ replaced by $(\bar s, v)$, where $v(t,x)=z(t,y(t,x))$; $z$ is a weak solution of \qq{2.5} with $(s,\tilde w)$ replaced by $(\bar s, w)$, where  $w(t,y)=u(t,x(t,y))$. Therefore $(u,v)\in\mathbb X^{\bar s}_{g,h}$ and solves \eqref{2.3} in $(0, \bar s]$. Applying the $L^p$ theory to \qq{2.5} with $(s,\tilde w)$ replaced by $(\bar s, w)$ and the uniqueness of weak solution, we have $z\in W^{1,2}_p(\Pi_{\bar s})\hookrightarrow C^{(1+\alpha)/2, 1+\alpha}(\Pi_{\bar s})$. Hence, $z(\bar s,\cdot)\in \overset{\circ}{W}{}^1_2(\Sigma)$. From the arguments in the above Steps $1,2,3$ we see that $\bar s$ depends only on $d_i,k_i,h_0$, $i=1,2$. By repeating this process finitely many times, the solution $(u,v)$ will be uniquely extended to $D^T_{g,h}$ and $(u,v)\in\mathbb X^T_{g,h}$.

Thanks to Lemma \ref{l2.2}, we have $u>0$ in $D^T_{g,h}$. And, it follows from the parabolic maximum principle for the strong solution that $v>0$ in $D^T_{g,h}$. Hence, we get \eqref{2.1}. Since $v>0$ in $D^T_{g,h}$ and $v(t,h(t))=v(t,g(t))=0$, we have $v_x(t,h(t))<0$ and $v_x(t,g(t))>0$ (see the proof of \cite[Theorem 1.1, pp.2597]{WZ-dcdsa18}). Recall $0\le v\le k_2$ and $f_2(t,x,u,v)\le Lv$. By using the similar arguments in the proof of \cite[Lemma 2.1]{WZjdde17} (cf. \cite[Lemma 2.1]{2-W18}), one can easily show that
 $$0<-v_x(t,h(t)),v_x(t,g(t))\le\max\left\{\frac 1{h_0},\ \ \sqrt{\frac{L}{2d_2}}, \
 \frac {\|v_0'\|_{C([-h_0,h_0])}}{k_2}\right\}=k_3.$$
This implies \eqref{2.2}. In view of \eqref{1.4} and the parabolic $L^p$ theory we have $v\in W^{1,2}_p(D^T_{g,h})$ for all $p>1$. The proof is complete.
\end{proof}

According to Lemma \ref{l2.3}, for any $T>0$ and $(g,h)\in\mathbb G^T\times\mathbb H^T$, there exists a unique $(u,v)=(u_{g,h},v_{g,h})\in\mathbb X_{g,h}^T$ that solves \eqref{2.3}, and \eqref{2.1} holds. For $0<t\le T$, define the mapping
 \[\mathcal{G}(g,h)=(\tilde g,\tilde h)\]
by
 \bess
 \begin{aligned}
&\tilde h(t)=h_0-\mu\int_0^t\!v_x(\tau,h(\tau)){\rm d}\tau+\rho\int_0^t\!\int_{g(\tau)}^{h(\tau)}\!\!\int_{h(\tau)}^\infty\!\!J(x,y)u(\tau,x){\rm d}y{\rm d}x{\rm d}\tau,\\[.1mm]
&\tilde g(t)=-h_0-\mu\int_0^t\!v_x(\tau,g(\tau)){\rm d}\tau-\rho\int_0^t\!\int_{g(\tau)}^{h(\tau)}\!\!\int_{-\infty}^{g(\tau)}\!\!J(x,y)u(\tau,x){\rm d}y{\rm d}x{\rm d}\tau.
\end{aligned}
 \eess
We shall show that $\mathcal{G}$ maps a suitable closed subset $\Gamma_T$ of $\mathbb G^T\times\mathbb H^T$ into itself and is a contraction mapping provided $T$ sufficiently small.

\begin{lemma}\lbl{l.a}\, There exists a closed subset $\Gamma_T\subset\mathbb G^T\times\mathbb H^T$ such that $\mathcal{G}(\Gamma_T)\subset\Gamma_T$.
\end{lemma}

\begin{proof}\, Let $(g,h)\in\mathbb G^T\times\mathbb H^T$. Then $\tilde g,\tilde h\in C^1([0,T])$ and for $0<t\le T$,
 \bess
\begin{aligned}
&\tilde h'(t)=-\mu v_x(t,h(t))+\rho\int_{g(t)}^{h(t)}\!\!\int_{h(t)}^\infty \!J(x,y)u(t,x){\rm d}y{\rm d}x,\\
&\tilde g'(t)=-\mu v_x(t,g(t))-\rho\int_{g(t)}^{h(t)}\!\!\int_{-\infty}^{g(t)}\!J(x,y)u(t,x){\rm d}y{\rm d}x.
\end{aligned}
 \eess
It follows that
 \bes
[\tilde h(t)-\tilde g(t)]'&=&-\mu\big[v_x(t,h(t))-v_x(t,g(t))\big]
+\rho\int_{g(t)}^{h(t)}\!\left[\int_{h(t)}^\infty+\int_{-\infty}^{g(t)}\right]\!J(x,y)u(t,x){\rm d}y{\rm d}x.\qquad
 \label{2.10}\ees
Taking
 \[0<\varepsilon_0<\min\left\{\bar\varepsilon,\ \frac{8\mu k_3}{\rho k_1}\right\},\ \ M=2h_0+\frac{\varepsilon_0}{4}, \ \ 0<T_0\le\frac{\varepsilon_0}{4\big(2\mu k_3+\rho k_1M\big)}\]
such that $h(T_0)-g(T_0)\le M$. Let $\br=\mu k_3+\rho k_1M$. Then, due to \eqref{2.1}, \eqref{2.2} and \eqref{2.10}, we have
 \[[\tilde h(t)-\tilde g(t)]'\le 2\mu k_3+\rho k_1[h(T_0)-g(T_0)]\le 2\mu k_3+\rho k_1M.\]
This implies
 \bes
 \tilde h(t)-\tilde g(t)\le 2h_0+t\big(2\mu k_3+\rho k_1M\big)\le M, \ \ t\in[0,T_0].\lbl{2.11}\ees
Similarly, we can show that
 \bes
 \tilde h'(t)\le \br,\ \ -\tilde g'(t)\le \br, \ \ t\in[0,T_0].
 \lbl{2.12}\ees
It is easily verified that
\bes
h(t)\in[h_0,h_0+\varepsilon_0/4],\ \ g(t)\in[-h_0-\varepsilon_0/4,-h_0],\ \ t\in[0,T_0].\label{2.13}
\ees
Since $(u,v)$ solves \eqref{2.3}, due to {\bf(f)}-{\bf(f2)} and \eqref{2.1} we have
\bess
\left\{\begin{aligned}
&u_t\ge -d_1u-Lu, &&(t,x)\in D^{T_0}_{g,h},\\
&u(t,g(t))=u(t,h(t))=0, &&0\le t\le T_0,\\
&u(0,x)=u_0(x), &&|x|\le h_0,
\end{aligned}\right.
 \eess
which implies that
 \bess
u(t,x)\ge e^{-(d_1+L)t}u_0(x)\ge e^{-(d_1+L)T_0}u_0(x),\ \ t\in(0,T_0], \ |x|\le h_0.
 \eess
This combined with \eqref{1.5} and \eqref{2.13} allows us to derive
\bess
\rho\int_{g(t)}^{h(t)}\!\!\int_{h(t)}^\infty\! J(x,y)u(t,x){\rm d}y{\rm d}x&\ge&\rho\int_{h(t)-\frac{\varepsilon_0}2}^{h(t)}\int_{h(t)}^{h(t)+\frac{\varepsilon_0}2} \!J(x,y)u(t,x){\rm d}y{\rm d}x\\
&\ge&\rho e^{-(d_1+L)T_0}\int_{h_0-\frac{\varepsilon_0}4}^{h_0}\int_{h_0
+\frac{\varepsilon_0}4}^{h_0+\frac{\varepsilon_0}2}\!J(x,y)u_0(x){\rm d}y{\rm d}x\\
&\ge&\frac{\varepsilon_0}{4}\delta_0\rho e^{-(d_1+L)T_0}\int_{h_0-\frac{\varepsilon_0}4}^{h_0}u_0(x){\rm d}x\\
&=:&\rho c_0,\ \ \ t\in(0,T_0].
\eess
Similarly,
\bess
-\rho\int_{g(t)}^{h(t)}\!\!\int_{-\infty}^{g(t)}\!J(x,y)u(t,x){\rm d}y{\rm d}x\le -\frac{\varepsilon_0}{4}\delta_0\rho e^{-(d_1+L)T_0}\int_{-h_0}^{-h_0+\frac{\varepsilon_0}4}\!u_0(x){\rm d}x=:-\rho c^*_0.
\eess
Thus, by \eqref{2.2},
\bes
\tilde h'(t)\ge \rho c_0,\ \ \tilde g'(t)\le-\rho c^*_0, \ \ \ t\in[0,T_0].
\lbl{2.14}\ees
Moreover, by the definitions of $\br,R(t)$ and the choice of $\varepsilon_0$, we know that
 \bess
 \bar R\le \mu k_3+2(h_0\rho k_1+\mu k_3)
\le\mu k_3+2(h_0\rho k_1+\mu k_3)e^{\rho k_1t}=R(t)
 \eess
for all $t\in[0,T_0]$. Noticing that
 $$\rho c_0\le\rho e^{-(d_1+L)T_0}\int_{h_0-\frac{\varepsilon_0}4}^{h_0}\int_{h_0
 +\frac{\varepsilon_0}4}^{h_0+\frac{\varepsilon_0}2}\!J(x,y)u_0(x){\rm d}y{\rm d}x\le \rho h_0 k_1,$$
one has
\[\rho c_0,\ \rho c^*_0\le \rho h_0 k_1< \br\le R(t),\ \ t\in[0,T_0].\]
For $0<T\le T_0$, we define
 \bess
 \Gamma_T=\{(g,h)\in\mathbb G^T\times\mathbb H^T: \rho c_0\le h'(t)\le \br,\ -\br\le g'(t)\le -\rho c^*_0,\ h(T)-g(T)\le M\}.
 \eess
It follows from the above analysis that $\mathcal{G}(\Gamma_T)\subset\Gamma_T$.
\end{proof}

In the following we show that $\mathcal{G}$ is a contraction mapping on $\Gamma_T$ when $T$ is small.

\begin{lemma}\label{l2.4} The mapping $\mathcal{G}$ is contraction on $\Gamma_T$ when $T$ is small.\end{lemma}

\begin{proof}\, For $(g_i,h_i)\in\Gamma_T$ with $0<T\le \min\{T_0,1\}$, let
\bess
&\Omega_T=D^T_{g_1,h_1}\cup D^T_{g_2,h_2},\ \ u_i=u_{g_i,h_i},\ \ v_i=v_{g_i,h_i},\ \ \mathcal{G}(g_i,h_i)=(\tilde g_i,\tilde h_i),\ \ i=1,2,&\\
&u=u_1-u_2,\ \ v=v_1-v_2,\ g=g_1-g_2,\ \ h=h_1-h_2,\ \ \tilde g=\tilde g_1-\tilde g_2,\ \ \tilde h=\tilde h_1-\tilde h_2.&
\eess
Note that $(u_i,v_i)\in\mathbb X^T_{g_i,h_i}$. By Lemma \ref{l2.3}, $v_i\in W^{1,2}_p(D^T_{g_i,h_i})$ with $p>3$. Make the zero extension of $u_i,v_i$ in $([0,T]\times \mathbb{R})\setminus D^T_{g_i,h_i}$ for $i=1,2$. It is easy to see that
\bes
|\tilde h'(t)|&\le&\mu|v_{1,x}(t,h_1(t))-v_{2,x}(t,h_2(t))|\nonumber\\[1mm]
&&+\rho\left| \int_{g_1(t)}^{h_1(t)}\!\!\int_{h_1(t)}^\infty\!J(x,y)u_1(t,x){\rm d}y{\rm d}x-\int_{g_2(t)}^{h_2(t)}\!\!\int_{h_2(t)}^\infty\!J(x,y)u_2(t,x){\rm d}y{\rm d}x \right|\nonumber\\[1mm]
&=:&\mu \phi_1(t)+\rho \phi_2(t).
 \label{2.15}\ees

{\it Step 1:\,The estimation of $\phi_1(t)$}. It follows from \eqref{2.3} that, for $i=1,2$,
 \bes
\left\{\begin{aligned}
&v_{i,t}=d_2v_{i,xx}+f_2(t,x,u_i,v_i), \ \ &&(t,x)\in D^T_{g_i,h_i},\\
&v_i(t,g_i(t))=v_i(t,h_i(t))=0, &&0\le t\le T,\\
&v_i(0,x)=0, &&|x|\le h_0.
\end{aligned}\right.
 \label{2.16}
 \ees
For $i=1,2$, let
 \[x_i(t,y)=\frac12[(h_i(t)-g_i(t))y+h_i(t)+g_i(t)],\]
and define
 \[w_i(t,y)=u_i(t,x_i(t,y)),\ \ z_i(t,y)=v_i(t,x_i(t,y)),\ \
 f^i_2(t,y,u,v)=f_2(t,x_i(t,y),u,v)\]
for $t\in[0,T],\,y\in\Sigma$ and $u,v\in\mathbb{R}^+$. Then \eqref{2.16} turns into
\bes
\left\{\begin{aligned}
&z_{i,t}=d_2\xi_i(t)z_{i,yy}+\zeta_i(t,y)z_{i,y}+f^i_2(t,y,w_i,z_i), &&0<t\le T,~|y|<1,\\
&z_i(t,-1)=z_i(t,1)=0, &&0\le t\le T,\\
&z_i(0,y)=v_0(h_0y)=:z_0(y), &&|y|\le 1,
\end{aligned}\right.
 \label{2.17}
 \ees
where $\xi_i(t)$ and $\zeta_i(t,y)$ are the same as $\xi(t)$ and $\zeta(t,y)$ in there $g,h$ are replaced by $g_i,h_i$. Making use of $(g_i,h_i)\in\Gamma_T$ and \eqref{2.1}, we have
\bes
\|\xi_i\|_{L^\infty((0,T))}\le 1/h_0^2,\ \ \|\zeta_i\|_{L^\infty(\Pi_T)}\le 2\br/h_0,\ \ \|f^i_2\|_{L^\infty(\Pi_T)}\le C_0\label{2.18}
\ees
for $i=1,2$, where $C_0$ depends only on $k_1,k_2$. By the parabolic $L^p$ theory, $z_i\in W^{1,2}_p(\Pi_T)$ and
 \bes
\|z_i\|_{W^{1,2}_p(\Pi_T)}\le C.\label{2.19}
 \ees
Same as \eqref{2.9} we have $[z_i,\,z_{i,y}]_{C^{\alpha/2,\alpha}(\Pi_T)}\le C_1$,
where $C_1$ is independent of $T^{-1}$. This implies
 \bess
 \|z_{i,y}\|_{C(\Pi_T)}\le \|z_0'(y)\|_{C(\Sigma)}+C_1T^{\alpha/2}
 \le \|z_0'(y)\|_{C(\Sigma)}+C_1.
 \eess
Extend $z_i(t,y)=0$ for $|y|\ge 1$. Then $z_{i,y}\in L^\infty([0,T]\times\mathbb{R})$ and
 \bes
 \|z_{i,y}\|_{L^\infty([0,T]\times\mathbb{R}})\le\|z_0'(y)\|_{C(\Sigma)}+C_1:=C_2.
 \lbl{2.20}\ees

Let $z=z_1-z_2$, $w=w_1-w_2$, $\xi=\xi_1-\xi_2$ and $\zeta=\zeta_1-\zeta_2$. It follows from \eqref{2.17} that
\bes
\left\{\begin{aligned}
&z_t-d_2\xi_1(t)z_{yy}-\zeta_1(t,y)z_{y}-a(t,y)z\\
 &\qquad=d_2\xi(t) z_{2,yy}+\zeta(t,y)z_{2,y}+b(t,y)+c(t,y)w, &&0<t\le T,~|y|<1,\\
&z(t,\pm 1)=0, &&0\le t\le T,\\
&z(0,y)=0, &&|y|\le 1,
\end{aligned}\right.
 \label{2.21}
 \ees
where
\bess
 a(t,y)&=&\dd\int_0^1f^1_{2,v}(t,y,w_1,z_2+(z_1-z_2)\tau){\rm d}\tau,\\[1mm] b(t,y)&=&f^1_2(t,y,w_1,z_2)-f^2_2(t,y,w_1,z_2),\\[1mm]
 c(t,y)&=&\dd\int_0^1f^2_{2,u}(t,y,w_2+(w_1-w_2)\tau,z_2){\rm d}\tau.
\eess
Note that $(g_i,h_i)\in\Gamma_T$. It follows that
 \bess
\|\xi\|_{L^\infty((0,T))}\le\frac A{h_0^4}\|g,\,h\|_{C([0,T])},\ \
\|\zeta\|_{L^\infty(\Pi_T)}\le\df{\br+A}{h_0^2}\|g,\,h\|_{C^1([0,T])}
 \eess
with $A=h_0+\varepsilon_0/4$, and
 \bess
 \|a,\,c\|_{L^\infty(\Pi_T)}\le L, \ \ \|b\|_{L^\infty(\Pi_T)}\le L^*\|g,\,h\|_{C([0,T])}.
  \eess
Recall \eqref{2.18}, \eqref{2.19}, applying the parabolic $L^p$ theory to \eqref{2.21}, one can obtain
 \bess
\|z\|_{W^{1,2}_p(\Pi_T)}\le C_3\big(\|g,\,h\|_{C^1([0,T])}+\|w\|_{C(\Pi_T)}\big),
 \eess
where $C_3$ depends on $h_0,\br,k_1,k_2,k_3,\varepsilon_0$. Same as \eqref{2.9}, one has
 \bes
[z]_{C^{\alpha/2,\alpha}(\Pi_T)}+[z_y]_{C^{\alpha/2,\alpha}(\Pi_T)}
&\le& C_4\big(\|g,\,h\|_{C^1([0,T])}+\|w\|_{C(\Pi_T)}\big),
 \lbl{2.22}\ees
where $C_4>0$ is independent of $T^{-1}$. We claim that, for $T$ small enough,
 \bes
 \|w\|_{C(\Pi_T)}\le C\big(\|u\|_{C(\overline\Omega_T)}+\|g,h\|_{C([0,T])}\big).
 \label{2.23}
 \ees
Because the proof of \eqref{2.23} is very long, it will be treated as a separate lemma (Lemma \ref{l2.x}). It follows from \eqref{2.22} and \eqref{2.23} that
 \bes
[z]_{C^{\alpha/2,\alpha}(\Pi_T)}+[z_y]_{C^{\alpha/2,\alpha}(\Pi_T)}
 \le C_5\big(\|g,\,h\|_{C^1([0,T])}+\|u\|_{C(\overline\Omega_T)}\big),
 \label{2.24}\ees
Noticing $z_{y}(0,1)=0$. One has, by \eqref{2.24},
 \bes
 |z_y(t,1)|_{C([0,T])}\le C_5T^{\alpha/2}\big(\|g,\,h\|_{C^1([0,T])}
 +\|u\|_{C(\overline\Omega_T)}\big).
 \lbl{2.25}\ees
As $h(0)=g(0)=0$, it is easy to see that
 \bes
|h(t)|\le t\|h'\|_{C([0,T])}\le t\|h\|_{C^1([0,T])},\ \ \ |g(t)|\le t\|g\|_{C^1([0,T])}.
 \label{2.26} \ees
As $v_{i,x}(t,h_i(t))=\frac{2z_{i,y}(t,1)}{h_i(t)-g_i(t)},\ i=1,2$,
it follows from \eqref{2.2} that $|z_{2,y}(t,1)|\le k_3M/2:=B$. Making use of \eqref{2.25} and \eqref{2.26} we have
\bes
 \phi_1(t)&=&|v_{1,x}(t,h_1(t))-v_{2,x}(t,h_2(t))|\nonumber\\[2mm]
 &=&\left|\df{2[z_{1,y}(t,1)-z_{2,y}(t,1)]}{h_1(t)-g_1(t)}
+2z_{2,y}(t,1)\df{g(t)-h(t)}{[h_1(t)-g_1(t)][h_2(t)-g_2(t)]}\right|\nonumber\\[1mm]
&\le&\df{1}{h_0}|z_y(t,1)|+2|z_{2,y}(t,1)|\df{|h(t)|+|g(t)|}{4h_0^2}\nonumber\\[1mm]
&\le&\df{1}{h_0}|z_y(t,1)|+2|z_{2,y}(t,1)|\df{t\|h\|_{C^1([0,T])}
+t\|g\|_{C^1([0,T])}}{4h_0^2}\nonumber\\[1mm]
&\le&\df{1}{h_0}C_5T^{\alpha/2}\big(\|g,\,h\|_{C^1([0,T])}
 +\|u\|_{C(\overline\Omega_T)}\big)+\df{B}{2h_0^2}T\|g,\,h\|_{C^1([0,T])}\nonumber\\[1mm]
 &\le& C_6T^{\alpha/2}\big(\|g,\,h\|_{C^1([0,T])}+\|u\|_{C(\bar\Omega_T)}\big).
\label{2.27}
\ees

{\it Step 2:\,The estimation of $\phi_2(t)$}. Inspiring by the arguments in \cite{CDLL, DWZ19} and using \eqref{2.26} we have
 \bes
\phi_2(t)&=&\left| \int_{g_1(t)}^{h_1(t)}\!\!\int_{h_1(t)}^\infty\! J(x,y)u_1(t,x){\rm d}y{\rm d}x-\int_{g_2(t)}^{h_2(t)}\!\!\int_{h_2(t)}^\infty\! J(x,y)u_2(t,x){\rm d}y{\rm d}x\right|\nonumber\\[1mm]
&\le&\int_{g_1(t)}^{h_1(t)}\!\!\int_{h_1(t)}^\infty\!J(x,y)|u(t,x)|{\rm d}y{\rm d}x\nonumber\\[1mm]
&&+\left|\left(\int_{g_1(t)}^{g_2(t)}\!\!\int_{h_1(t)}^\infty+\int_{h_2(t)}^{h_1(t)}
\!\!\int_{h_1(t)}^\infty
+\int_{g_2(t)}^{h_2(t)}\!\!\int_{h_1(t)}^{h_2(t)}\right)J(x,y)u_2(t,x){\rm d}y{\rm d}x\right|\nonumber\\[1mm]
&\le&3h_0\|u\|_{C(\bar\Omega_T)}+k_1\|g\|_{C([0,T])}+2k_1\|h\|_{C([0,T])}\nonumber\\[1mm]
 &\le& C_7\big(\|u\|_{C(\bar\Omega_T)}+T\|g,\,h\|_{C^1([0,T])}\big).
 \label{2.28}
\ees

{\it Step 3:\, The estimation of $\|u\|_{C(\bar\Omega_T)}$}. Fixed $(s,x)\in\Omega_T$.

Case 1: $x\in(g_1(s),h_1(s))\setminus(g_2(s),h_2(s))$. In this case, either $g_1(s)<x\le g_2(s)$ or $h_2(s)\le x<h_1(s)$ and $u_2(s,x)=v_2(s,x)=0$. For $h_0<h_2(s)\le x<h_1(s)$, there is $0<s_1<s$ such that $x=h_1(s_1)$. Clearly, $h_2(t)\le h_2(s)\le x=h_1(s_1)<h_1(s)$ and $g_1(t)<h_1(s_1)=x\le h_1(t)$ for $t\in[s_1,s]$. Hence, $u_2(t,x)=0$ for $t\in[s_1,s]$ and $u_1(s_1,x)=0$. Integrating the equation of $u_1$ from $s_1$ to $s$ gives
 \bess
|u(s,x)|&=&u_1(s,x)=\int_{s_1}^s \left(d_1\int_{g_1(t)}^{h_1(t)}\!J(x,y)u_1(t,y){\rm d}y-d_1u_1+f_1(t,x,u_1,v_1) \right){\rm d}t\\[1mm]
&\le&(s-s_1)(d_1+L)k_1\\[1mm]
&\le&(\rho c_0)^{-1}[h_1(s)-h_1(s_1)](d_1+L)k_1\\[1mm]
&\le&(\rho c_0)^{-1}(d_1+L)k_1[h_1(s)-h_2(s)]\\[1mm]
&\le&C_8\|h_1-h_2\|_{C([0,T])}.
\eess
When $g_1(s)<x\le g_2(s)$, by using the similar arguments, it is easy to derive that $|u(s,x)|=u_1(s,x)\le C_8'\|g\|_{C([0,s])}$. Therefore,
$|u(s,x)|\le C_9\|g,\,h\|_{C([0,s])}$ with $C_9=\max\{C_8,C_8'\}$.
This combined with \eqref{2.26} allows us to derive
 \bes
 |u(s,x)|\le C_9T\|g,\,h\|_{C^1([0,s])}.
 \label{2.29}\ees

Case 2: $x\in(g_2(s),h_2(s))\setminus(g_1(s),h_1(s))$. Parallel to the case 1 we have \eqref{2.29}.

Case 3: $x\in(g_1(s),h_1(s))\cap(g_2(s),h_2(s))$. If $x\in(g_1(t),h_1(t))\cap(g_2(t),h_2(t))$ for all $0<t<s$, then
 \bes
u_t(t,x)&=&u_{1t}(t,x)-u_{2t}(t,x)\nonumber\\
&=&d_1\int_{g_1(t)}^{h_1(t)}J(x,y)u(t,y){\rm d}y+d_1\left(\int_{g_1(t)}^{g_2(t)}
 +\int_{h_2(t)}^{h_1(t)}\right)J(x,y)u_2(t,y){\rm d}y\nonumber\\[1mm]
&&-d_1u(t,x)+f_1(t,x,u_1,v_1)-f_1(t,x,u_2,v_2).
  \label{2.30}\ees
Notice that
\[|f_1(t,x,u_1,v_1)-f_1(t,x,u_2,v_2)|\le L(|u|+|v|),\]
and $u(0,x)=u_1(0,x)-u_2(0,x)=0$. Integrating \eqref{2.30} from $0$ to $s$ yields
 \bes
|u(s,x)|&\le& T\big((2d_1+L)\|u\|_{C(\bar\Omega_T)}+d_1k_1\|J\|_\infty\|g,\,h\|_{C([0,s])}\big)
+L\int_0^s|v(t,x)|{\rm d}t\nonumber\\
 &\le& TC_{10}\big(\|u\|_{C(\bar\Omega_T)}+\|g,\,h\|_{C^1([0,T])}\big)
+L\int_0^s|v(t,x)|{\rm d}t.\label{2.31}
\ees

If there is $0<t<s$ such that $x\notin(g_1(t),h_1(t))\cap(g_2(t),h_2(t))$, then we can choose the largest $t_0\in(0,t)$ such that
  \bes
x\in(g_1(t),h_1(t))\cap(g_2(t),h_2(t)),\ \ \forall\ t_0<t\le s,
 \label{2.32}\ees
and
\bess
x\in(g_1(t_0),h_1(t_0))\setminus(g_2(t_0),h_2(t_0)),\ \ {\rm or}\ \ x\in(g_2(t_0),h_2(t_0))\setminus(g_1(t_0),h_1(t_0)).
\eess
It follows from the conclusions of Case 1 and Case 2 that $|u(t_0,x)|\le C_9\|g,\,h\|_{C([0,s])}$. Thus,
 \[|u(t_0,x)|\le C_9s\|g,\,h\|_{C^1([0,s])}\le C_9T\|g,\,h\|_{C^1([0,T])}\]
by \eqref{2.26}. Note that \eqref{2.30} holds for any $t_0<t\le s$ due to \eqref{2.32}. Integrating \eqref{2.30} from $t_0$ to $s$ we have
  \bes
|u(s,x)|&\le& |u(t_0,x)|+T\big((2d_1+L)\|u\|_{C(\bar\Omega_T)}+d_1k_1\|J\|_\infty\|g,\,h\|_{C([0,s])}\big)
+L\int_{t_0}^s|v(t,x)|{\rm d}t\nonumber\\
&\le& C_{11}T\big(\|u\|_{C(\bar\Omega_T)}+\|g,\,h\|_{C^1([0,T])}\big)
+L\int_{t_0}^s|v(t,x)|{\rm d}t.\label{2.33}
 \ees

Now we estimate $\dd\int_{t_0}^s|v(t,x)|{\rm d}t$ and $\dd\int_0^s|v(t,x)|{\rm d}t$.
Let
 \[y_i=y_i(t,x)=\frac{2x-h_i(t)-g_i(t)}{h_i(t)-g_i(t)},\ \ i=1,2.\]
Then
 \[x=\frac{(h_i(t)-g_i(t))y_i+h_i(t)+g_i(t)}2,\]
and due to \eqref{2.32} we have $y_i(t,x)\in\Sigma
$. Moreover,
 \bes
\|y_1(\cdot,x)-y_2(\cdot,x)\|_{C([t_0,s])}
\le\frac{ h_0+\varepsilon_0/4}{h_0^2}\|g,\,h\|_{C([0,T])}=\frac{ A}{h_0^2}\|g,\,h\|_{C([0,T])}.
 \label{2.34}\ees
Clearly, $z_i(t,y_i)=v_i(t,x)$ for $t_0<t\le s$. Note that $z(0,y)=z_1(0,y)-z_2(0,y)=0$, we have that, for any $(t,y)\in\Pi_T$,
 \[|z(t,y)|=|z(t,y)-z(0,y)|\le t^{\alpha/2}[z]_{C^{\alpha/2, \alpha}(\Pi_T)}.\]
And so $\|z\|_{C(\Pi_T)}\le T^{\alpha/2}[z]_{C^{\alpha/2, \alpha}(\Pi_T)}$. Thanks to \eqref{2.20}, \eqref{2.24} and \eqref{2.34}, it educes that
 \bes
\int_{t_0}^s|v(t,x)|{\rm d}t&=&\int_{t_0}^s|z_1(t,y_1)-z_2(t,y_2)|{\rm d}t\nonumber\\[1mm]
&\le&\int_{t_0}^s|z_1(t,y_1)-z_2(t,y_1)|{\rm d}t+\int_{t_0}^s|z_2(t,y_1)-z_2(t,y_2)|{\rm d}t\nonumber\\[1mm]
&\le&T\|z\|_{C(\Pi_T)}+\int_{t_0}^s|y_1-y_2|\|z_{2,y}\|_{L^\infty([0,T]\times\mathbb{R})}{\rm d}t\nonumber\\[1mm]
 &\le&T\|z\|_{C(\Pi_T)}+T\|y_1-y_2\|_{C([t_0,s])}\|z_{2,y}\|_{L^\infty([0,T]\times\mathbb{R})}\nonumber\\[1mm]
&\le&C_5T^{1+\alpha/2}\big(\|g,\,h\|_{C^1([0,T])}+\|u\|_{C(\overline\Omega_T)}\big)
+\frac{A C_2}{h_0^2}T\|g,\,h\|_{C([0,T])}\nonumber\\[1mm]
 &\le& C_{12}T\big(\|g,\,h\|_{C^1([0,T])}+\|u\|_{C(\overline\Omega_T)}\big).
 \lbl{2.35}\ees
Similarly, one can find $C_{13}>0$ such that
 \bes
\int_0^s|v(t,x)|{\rm d}t\le C_{13}T\big(\|g,\,h\|_{C^1([0,T])}+\|u\|_{C(\overline\Omega_T)}\big).
 \lbl{x}\ees
Substituting the estimations \qq{2.35} and \qq{x} into \eqref{2.33} and \eqref{2.31}, respectively,
we have
  \bes
|u(s,x)|\le C_{14}T\big(\|g,\,h\|_{C^1([0,T])}+\|u\|_{C(\overline\Omega_T)}\big).
 \lbl{2.36}\ees
The estimates \eqref{2.29} and \eqref{2.36} show that, for any case, the following holds:
 \bess
|u(s,x)|\le C_{14}'T\big(\|g,\,h\|_{C^1([0,T])}+\|u\|_{C(\overline\Omega_T)}\big).
 \eess
The arbitrariness of $(s,t)\in\Omega_T$ implies
 \bes
\|u\|_{C(\overline\Omega_T)}\le 2C_{14}'T\|g,\,h\|_{C^1([0,T])}.
 \lbl{2.37}\ees
provided $C_{14}'T\le1/2$.

\vskip 4pt{\it Step 4}:\, Inserting \eqref{2.37} into \eqref{2.27}, \eqref{2.28} we get
 \bess
\mu \phi_1(t)+\rho \phi_2(t)&\le&C_{15}T^{\alpha/2}\|g,\,h\|_{C^1([0,T])},
\ \ \forall \ 0<t\le T.
 \eess
This combined with \eqref{2.15} implies
  \[|\tilde h'(t)|\le C_{15}T^{\alpha/2}\|g,\,h\|_{C^1([0,T])},\ \ \forall \ 0<t\le T.\]
Similarly,
  \[|\tilde g'(t)|\le C_{16}T^{\alpha/2}\|g,\,h\|_{C^1([0,T])},\ \ \forall \ 0<t\le T.\]
Moreover, as $\tilde g(0)=\tilde h(0)=0$, it is easy to deduce that
 \bess
 \|\tilde g(t),\,\tilde h(t)\|_{C^1([0,T])}
\le2(C_{15}+C_{16})T^{\alpha/2}\|g,\,h\|_{C^1([0,T])}
\le \df12\|g,\,h\|_{C^1([0,T])}
 \eess
when $T$ is small. Hence, $\mathcal{G}$ is a contraction mapping on $\Gamma_T$ when $T$ is small.\end{proof}

\begin{lemma}\lbl{l2.x}\,The estimate \eqref{2.23} holds.
\end{lemma}

\begin{proof}\, To save space, let's assume $d_1=1$ here.
For the fixed $(\tau,y)\in\Pi_T$, we set
 \[x_i=x_i(\tau,y)=\frac12[(h_i(\tau)-g_i(\tau))y+g_i(\tau)+h_i(\tau)],\ \ i=1,2.\]
Then, $w_i(\tau,y)=u_i(\tau,x_i)$, $x_i\in[g_i(\tau),h_i(\tau)]$. The direct calculation yields
 \bes
x_1-x_2=\df{(h_2(\tau)-x_2)(g_1(\tau)-g_2(\tau))}{h_2(\tau)-g_2(\tau)}
+\df{(x_2-g_2(\tau))(h_1(\tau)-h_2(\tau))}{h_2(\tau)-g_2(\tau)},
 \label{2.38}\ees
which combined with the definition of $\Gamma_T$ and \eqref{2.13} implies
  \[|x_1-x_2|\le\df{M\varepsilon_0}{4h_0}\le\df{3\varepsilon_0}{4}<h_0.\]
Hence, one of the following four cases must happen:
\[x_1,x_2\in[-h_0,h_1(\tau)]; \ \ x_1,x_2\in[-h_0,h_2(\tau)]; \ \ x_1,x_2\in[g_1(\tau),h_0]; \ \ x_1,x_2\in[g_2(\tau),h_0].\]
Without loss of generality we may suppose that $h_1(\tau)\ge h_2(\tau)$ and $x_1,x_2\in[-h_0,h_1(\tau)]$. For other cases, one can handle by the same way.
Similar to Step 3 in the proof of Lemma \ref{l2.3}, for this fixed $\tau$ and any $x\in[g_1(\tau),h_1(\tau)]$, we define
 \begin{align*}
\tau_x=&\;\left\{\begin{aligned}
&\tau_{x,g_1}& &\mbox{ if }\ x\in[g_1(\tau),-h_0), \ x=g_1(\tau_{x,g_1}),\\
&0& &\mbox{ if } \ |x|\le h_0,\\
&\tau_{x,h_1}& &\mbox{ if } \ x\in(h_0,h_1(\tau)], \ x=h_1(\tau_{x,h_1}).
\end{aligned}
\right.
\end{align*}
As $x_i\in[-h_0,h_1(\tau)]$, we have $\tau_{x_i}=\tau_{x_i,h_1}$ or $\tau_{x_i}=0$, and $0\le \tau_{x_i}\le \tau$, $i=1,2$. It is easy to get
 \bes
|w_1(\tau,y)-w_2(\tau,y)|&\le&|u_1(\tau,x_1)-u_1(\tau,x_2)|+|u_1(\tau,x_2)-u_2(\tau,x_2)|\nonumber\\[1mm]
&\le&|u_1(\tau,x_1)-u_1(\tau,x_2)|+\|u\|_{C(\overline\Omega_T)}.
 \lbl{2.39}\ees

We estimate $|u_1(\tau,x_1)-u_1(\tau,x_2)|$. Integrating the differential equation of $u_1$ from $\tau_{x}$ to $\tau$ gives
 \bess
u_1(\tau,x)&=&u_1(\tau_{x},x)+\int_{\tau_{x}}^\tau
\left(\int_{g_1(s)}^{h_1(s)}\!\!J(x,y)u_1(s,y){\rm d}y-u_1(s,x)+ f_1(s,x,u_1,v_1)\right){\rm d}s.
 \eess
Denote $\tau_i=\tau_{x_i}$, $i=1,2$. Then $\tau_i$ depends on $x_i$. Without loss of generality we assume $\tau_1\ge \tau_2$. Thus, for $\tau_1\le t\le \tau$,
  \bess
|u_1(t,x_1)-u_1(t,x_2)|
&\le&|u_1(\tau_1,x_1)-u_1(\tau_2,x_2)|+\! \int_{\tau_1}^t\!\int_{g_1(s)}^{h_1(s)}\!|J(x_1,y)-J(x_2,y)|u_1(s,y){\rm d}y{\rm d}s\\[1mm]
&&+\int_{\tau_2}^{\tau_1}\!\!\int_{g_1(s)}^{h_1(s)}\!\!J(x_2,y)u_1(s,y){\rm d}y{\rm d}s
+\int_{\tau_1}^t|u_1(s,x_1)-u_1(s,x_2)|{\rm d}s\\[1mm]
&& +\int_{\tau_2}^{\tau_1}u_1(s,x_2){\rm d}s+\int_{\tau_2}^{\tau_1}|f_1(s,x_2,u_1(s,x_2),v_1(s,x_2))|{\rm d}y{\rm d}s\\[1mm]
&&+\int_{\tau_1}^t|f_1(s,x_1,u_1(s,x_1),v_1(s,x_1))
-f_1(s,x_2,u_1(s,x_2),v_1(s,x_2))|{\rm d}y{\rm d}s.
 \eess
It follows from the conditions {\bf(f)} and {\bf(f3)} that
 \bess
 &&|f_1(s,x_2,u_1(s,x_2),v_1(s,x_2))|\le L|u_1(s,x_2)|\le Lk_1,\\[1mm]
 &&|f_1(s,x_1,u_1(s,x_1),v_1(s,x_1))-f_1(s,x_2,u_1(s,x_2),v_1(s,x_2))|\\[1mm]
 &\le& L^*|x_1-x_2|+L\big(|u_1(s,x_1)-u_1(s,x_2)|+|v_1(s,x_1)-v_1(s,x_2)|\big).
 \eess
 As $g_i,h_i$ satisfy \eqref{2.11}, i.e., $h_i(\tau)-g_i(\tau)\le M$ in $[0,T]$, using the condition {\bf(J1)} we have
 \bes
 |u_1(t,x_1)-u_1(t,x_2)|
&\le& |u_1(\tau_1,x_1)-u_1(\tau_2,x_2)|+Tk_1ML(J)|x_1-x_2|
+k_1|\tau_1-\tau_2|\nonumber\\[1mm]
&&+T\|u_1(\cdot,x_1)-u_1(\cdot,x_2)\|_{C([\tau_1,t])}+k_1|\tau_1-\tau_2|\nonumber\\[1mm]
&&+Lk_1|\tau_1-\tau_2|+TL^*|x_1-x_2|+TL\|u_1(\cdot,x_1)-u_1(\cdot,x_2)\|_{C([\tau_1,t])}\nonumber\\[1mm]
&&+L\int_{\tau_1}^{t}|v_1(s,x_1)-v_1(s,x_2)|{\rm d}s\nonumber\\[1mm]
&\le&|u_1(\tau_1,x_1)-u_1(\tau_2,x_2)|+L\int_{\tau_1}^{t}|v_1(s,x_1)-v_1(s,x_2)|{\rm d}s\nonumber\\[1mm]
 &&+C\big(T|x_1-x_2|+|\tau_1-\tau_2|+T\|u_1(\cdot,x_1)
 -u_1(\cdot,x_2)\|_{C([\tau_1,t])}\big)
 \label{2.40} \ees
for all $\tau_1\le t\le \tau$. From \eqref{2.38}, one has
 \bes
|x_1-x_2|\le\df{M}{2h_0}\|g,h\|_{C([0,T])}.
 \label{2.41}\ees

In the following we estimate $|\tau_1-\tau_2|$ and $|u_1(\tau_1,x_1)-u_1(\tau_2,x_2)|$.

{\it Case 1:} $\tau_i>0$ for $i=1,2$. In this case, it is clear that
$u_1(\tau_1,x_1)=u_1(\tau_2,x_2)=0$. On the other hand, since $(g_1, h_1)\in\Gamma_T$, we have $h'_1\ge \rho c_0$ in $[0,\tau]$, and so
\bess
|\tau_1-\tau_2|\le (\rho c_0)^{-1}|h_1(\tau_1)-h_1(\tau_2)|&=&(\rho c_0)^{-1}|x_1-x_2|,
\eess

{\it Case 2:} $\tau_1>0$ and $\tau_2=0$. Then $x_2\in[-h_0,h_0]$, $x_1>h_0$, $u_1(\tau_1,x_1)=0$. Let $L(u_0)$ be the Lipschitz constant of $u_0$. It follows that
 \bess
 &|\tau_1-\tau_2|=|\tau_1-0|\le(\rho c_0)^{-1}|h_1(\tau_1)-h_1(0)|
=(\rho c_0)^{-1}|x_1-h_0|\le(\rho c_0)^{-1}|x_1-x_2|,&\\
 &|u_1(\tau_1,x_1)-u_1(\tau_2,x_2)|=|0-u_0(x_2)|=|u_0(h_0)-u_0(x_2)|\le L(u_0)|h_0-x_2|\le L(u_0)|x_1-x_2|.&
 \eess

{\it Case 3:} $\tau_1=\tau_2=0$, i.e., $x_1,x_2\in[-h_0,h_0]$. Then $|\tau_1-\tau_2|=0$,
and
 \bess
 |u_1(\tau_1,x_1)-u_1(\tau_2,x_2)|=|u_0(x_1)-u_0(x_2)|\le L(u_0)|x_2-x_1|.
\eess

In a word,
  \bes
|\tau_1-\tau_2|+|u_1(\tau_1,x_1)-u_1(\tau_2,x_2)|\le
[(\rho c_0)^{-1}+L(u_0)]|x_1-x_2|.
 \label{2.42}\ees

Now we estimate $\dd\int_{\tau_1}^t|v_1(s,x_1)-v_1(s,x_2)|{\rm d}s$. Let $y_i=\dd\frac{2x_i-g_1(\tau)-h_1(\tau)}{h_1(\tau)-g_1(\tau)}$. Then $z_1(\tau,y_i)=v_1(\tau,x_i)$. Similar to the derivation of \eqref{2.35} we have
 \bess
 \int_{\tau_1}^t|v_1(s,x_1)-v_1(s,x_2)|{\rm d}s&=&\int_{\tau_1}^t|z_1(s,y_1)-z_1(s,y_2)|{\rm d}s\\[1mm]
&\le&T|y_1-y_2|\cdot\|z_{1,y}\|_{L^\infty([0,T]\times\mathbb{R})}\\[1mm]
&\le&T C_{17}|x_1-x_2|.
 \eess
Substituting this and \eqref{2.42} into \eqref{2.40} and using \eqref{2.41}, it yields that, for $\tau_1\le t\le \tau$,
 \bess
 |u_1(t,x_1)-u_1(t,x_2)|\le C_{18}\big(\|g,h\|_{C([0,T])}+T\|u_1(\cdot,x_1)
 -u_1(\cdot,x_2)\|_{C([\tau_1,t])}\big).
 \eess
Thus we have
 \bess
 \|u_1(\cdot,x_1)-u_1(\cdot,x_2)\|_{C([\tau_1,\tau])}\le C_{18}\big(\|g,h\|_{C([0,T])}+T\|u_1(\cdot,x_1)-
 u_1(\cdot,x_2)\|_{C([\tau_1,\tau])}\big)
 \eess
Taking $T$ small such that $C_{18}T<1/2$, then
 \bess
 |u(\tau,x_1)-u(\tau,x_2)|\le\|u_1(\cdot,x_1)-u_1(\cdot,x_2)\|_{C([\tau_1,\tau])}
 \le 2C_{18}\|g,h\|_{C([0,T])}.
  \eess
Substituting this into \eqref{2.39} and by the arbitrariness of $(\tau,y)\in\Pi_T$, we get \eqref{2.23}  immediately.
 \end{proof}

\begin{proof}[Proof of Theorem \ref{th2.1}]

{\it Step 1:\, Local existence and uniqueness}. By Lemma \ref{l.a} and Lemma \ref{l2.4} we see that $\mathcal{G}(\Gamma_T)\subset\Gamma_T$ and $\mathcal{G}$ is a contraction mapping on $\Gamma_T$ when $T$ is small. The {\it Contraction Mapping Theorem} shows that problem \eqref{1.3} admits a unique solution $(\hat u,\hat v,\hat g,\hat h)$ with $(\hat g,\hat h)\in\Gamma_T$. This solution is the unique solution of \eqref{1.3} if we can prove that $(g,h)\in\Gamma_T$ holds for any solution $(u,v,g,h)$ of \eqref{1.3} defined for $t\in(0,T]$. Moreover, from the above arguments we see that $(\hat u,\hat v,\hat g,\hat h)$ satisfies \eqref{2.1} and \eqref{2.2}.

Let $(u,v,g,h)$ be an arbitrary solution of \eqref{1.3} defined in $(0,T]$. It follows that
 \bess
\begin{aligned}
&h'(t)=-\mu v_x(t,h(t))+\rho\int_{g(t)}^{h(t)}\!\!\int_{h(t)}^\infty\! J(x,y)u(t,x){\rm d}y{\rm d}x,\\[1mm]
&g'(t)=-\mu v_x(t,g(t))-\rho\int_{g(t)}^{h(t)}\!\!\int_{-\infty}^{g(t)}\! J(x,y)u(t,x){\rm d}y{\rm d}x.
\end{aligned}
 \eess
It is easy to see from the above discussions that \eqref{2.1} and \eqref{2.2} hold. And hence
 \[[h(t)-g(t)]'\le 2\mu k_3+\rho k_1(h(t)-g(t));\ \ \ 0<-g'(t), h'(t)\le\mu k_3+\rho k_1(h(t)-g(t)).\]
The first inequality in the above implies $h(t)-g(t)\le2[h_0+\mu k_3/(\rho k_1)]e^{\rho k_1t}$. So we have
 \bess
 &[h(t)-g(t)]'\le 2\mu k_3+2(\rho k_1h_0+\mu k_3)e^{\rho k_1t},&\\[1mm]
 &0<h'(t),-g'(t)\le \mu k_3+2(\rho k_1h_0+\mu k_3)e^{\rho k_1t}=R(t).&\eess
Therefore,
 \[h(t)-g(t)\le 2h_0+t\kk(2\mu k_3+2(\rho k_1h_0+\mu k_3)e^{\rho k_1t}\rr), \ \ \forall \ 0<t\le T.\]
Shrink $T$ small enough such that $T\big[2\mu k_3+2(\rho k_1h_0+\mu k_3)e^{\rho k_1T}\big]\le{\varepsilon_0}/{4}$.
Then $h(t)-g(t)\le M$ for $t\in[0,T]$. Furthermore, by using the proofs of \eqref{2.12} and \eqref{2.14}, one can show that $\rho c_0\le h'(t)\le \br$ and $-\br\le g'(t)\le-\rho c_0^*$ in $(0,T]$. Thus $(g,h)\in\Gamma_T$.

{\it Step 2:\, Global existence and uniqueness}. Assume that \qq{1.8} holds. From Step 1, we know that the system \eqref{1.3} admits a unique solution $(u,v,g,h)$ in some interval $(0,T]$.

Let $z(t,y)=v(t,x(t,y))$ and consider the problem
\bes
\left\{\begin{aligned}
&z_t=d_2\xi(t)z_{yy}+\zeta(t,y)z_y+f^*_2(t,y,w,z), &&0<t\le T,~|y|<1,\\[1mm]
&z(t,\pm 1)=0, &&0\le t\le T,\\[1mm]
&z(0,y)=v_0(h_0y)=:z_0(y), &&|y|\le 1,
\end{aligned}\right.
 \label{2.43} \ees
where $w(t,y)=u(t,x(t,y))$, $f^*_2(t,y,w,z)=f_2(t,x(t,y),w,z)$. As $z_0(y)\in W^2_p(\Sigma)$, same as the above, $z\in W^{1,2}_p(\Pi_T)\hookrightarrow C^{(1+\alpha)/2, 1+\alpha}(\Pi_T)$. Then $v_x\in C^{\alpha/2, \alpha}(\ol D^T_{g,h})$. This combined with the assumptions {\bf(f)} and {\bf(f3)} implies that the function $F_1(t,x,u)=f_1(t,x,u,v(t,x))$ is differentiable with respect to $x$. Note that $u$ satisfies
\bess
\left\{\begin{aligned}
&u_t=d_1\int_{g(t)}^{h(t)} J(x,y)u(t,y){\rm d}y-d_1u+f_1(t,x,u,v(t,x)), &&t_x<t\le T,~g(t)<x<h(t),\\[1mm]
&u(t_x,x)=\tilde u_0(x), &&g(T)<x<h(T),
\end{aligned}\right.
 \eess
where
 \begin{align*}
\tilde u_0(x)=\left\{\begin{aligned}
&0,& &|x|>h_0,\\
&u_0(x),& &|x|\le h_0,
\end{aligned}\right. \quad t_x=\left\{\begin{aligned}
&t_{x,g}& &\mbox{if }\ x\in[g(T),-h_0), \ x=g(t_{x,g}),\\
&0& &\mbox{if } \ |x|\le h_0,\\
&t_{x,h}& &\mbox{if } \ x\in(h_0,h(T)], \ x=h(t_{x,h}).
\end{aligned}
\right.
  \end{align*}
View $G(t,x)=\int_{g(t)}^{h(t)} J(x,y)u(t,y){\rm d}y$ as a known function. Then for $t\in[0,T]$, $t_x, u_0(x)$ and $G(t,x)$ are Lipschitz continuous in $x\in[g(t),h(t)]$. Using the continuous dependence of the solution with respect to the parameters we can show that for $t\in[0,T]$, $u(t,x)$ is Lipschitz continuous in $x\in[g(t),h(t)]$. Clearly, $u_t\in C(\ol D^T_{g,h})$. This implies $u\in C^{1,1-}(\ol D^T_{g,h})$ and hence $w\in C^{1,1-}(\ol\Pi_T)$.

It is easy to see that the function
 \[\int_{g(t)}^{h(t)}\!\int_{h(t)}^\infty\! J(x,y)u(t,x){\rm d}y{\rm d}x\]
of $t$ is differentiable. So $h'(t)\in C^{\alpha/2}([0,T])$ as $v_x(t,h(t))\in C^{\alpha/2}([0,T])$. Similarly, $g'(t)\in C^{\alpha/2}([0,T])$. Set  $F_2(t,y,z)=f_2^*(t,y,w(t,y),z)$. Then, by using {\bf(f4)} (or \qq{1.8}), there hold
 \[\xi\in C^{\alpha/2}([0,T]), \ \ \zeta(\cdot,\cdot), \,F_2(\cdot,\cdot,z)\in C^{\alpha/2, \alpha}(\Pi_T).\]
By the interior Schauder theory we have $z\in C^{1+\alpha/2, 2+\alpha}([\ep,T]\times\Sigma)$ with $0<\ep<T$, which implies $v(T,x)\in C^2([g(T), h(T)])$.

Recall that $u(T,x)$ is Lipschitz continuous in $x\in[g(T),h(T)]$. We can take $(u(T,x),v(T,x))$ as an initial function and $[g(T),h(T)]$ as the initial habitat and then use Step 1 to extend the solution from $t=T$ to some $T'>T$. Assume that $(0,T_0)$ is the maximal existence interval of $(u,v,g,h)$ obtained by such extension process. We shall prove that $T_0=\infty$. Assume on the contrary that $T_0<\infty$.

Since $h',-g'>0$ in $(0,T_0)$, we can define $h(T_0)=\dd\lim_{t\rightarrow T_0}h(t)$ and $g(T_0)=\dd\lim_{t\rightarrow T_0}g(t)$. By the above arguments,
\[h(T_0)-g(T_0)\le 2h_0+T_0\big(2\mu k_3+2(\rho k_1h_0+\mu k_3)e^{\rho k_1T_0}\big).\]
In view of $0<-v_x(t,h(t)),v_x(t,g(t))\le k_3,0<u\le k_1, 0<v\le k_2$ for $t\in(0,T_0)$, $h',g'\in L^\infty((0,T_0))$. Making use of Sobolev embedding theorem : $W^1_\infty((0,T_0))\hookrightarrow C([0,T_0])$, we have $g,h\in C([0,T_0])$ with $g(T_0),h(T_0)$ defined as above. It follows from the parabolic $L^p$ theory and Sobolev embedding theorem that $v\in C^{(1+\alpha)/2,1+\alpha}(\ol D^{T_0}_{g, h})$. These facts show that the first differential equation holds for $0\le t\le T_0$. Similar to the above, $u\in C^{1,1-}(\ol D^{T_0}_{g,h})$, $g',\,h'\in C^{\alpha/2}([0,T_0])$. Consider the problem \qq{2.43} with $T$ replaced by $T_0$. Same as above, we can show that \qq{2.43} has a unique solution
$z\in W^{1,2}_p(\Pi_{T_0})\cap C^{1+\alpha/2,2+\alpha}([\ep, T_0]\times\Sigma)$. Consequently, $v(T_0,x)\in C^2([g(T_0), h(T_0)])$.

Due to $u(t,h(t))=v(t,h(t))=0$ in $[0,T_0)$, it is easy to see that $u(T_0,h(T_0))=v(T_0,h(T_0))=0$. Moreover, by the parabolic maximum principle and Lemma \ref{l2.2} we have $u(T_0,x)>0,v(T_0,x)>0$ for $x\in(g(T_0),h(T_0))$.

Therefore, we may treat $(u(T_0,x),v(T_0,x))$ as an initial function and $[g(T_0),h(T_0)]$ as the initial habitat and apply Step 1 to show that the solution of \eqref{1.3} can be extended to some $(0,\hat T)$ with $\hat T>T_0$. This contradicts the definition of $T_0$. Hence, $T_0=\infty$.

It follows from the above arguments that $(g, h)\in\mathbb G^T\times\mathbb H^T$, $(u,v)\in\mathbb X^T_{g,h}$, and $(u, v,g,h)$ satisfies \eqref{2.1}, \eqref{2.2} and \eqref{b.3}. The proof is end.
\end{proof}


\begin{thebibliography}{99}
\bibliographystyle{siam}
\setlength{\baselineskip}{15pt}

\vspace{-1.5mm}\bibitem{GW12}J. Guo and C. H. Wu, {\it On a free boundary problem for a two-species weak competition system}, J. Dyn. Diff. Equat., \textbf{24}
(2012) 873-895.

\vspace{-1.5mm}\bibitem{WZjdde14} 
M. X. Wang and J. F.  Zhao, {\it Free boundary problems for a Lotka-Volterra competition system}, {J. Dyn. Diff. Equat.}, \textbf{26}(3)(2014), 655-672.

\vspace{-1.5mm}\bibitem{WZjdde17} M. X. Wang and J. F. Zhao, {\it A free boundary problem for the  predator-prey model with double free boundaries}, J. Dyn. Diff. Equat., {\bf 29}(3)(2017), 957-979.

\vspace{-1.5mm}\bibitem{DWZ19}Y. H. Du, M. X. Wang and M. Zhao, {\it Two species nonlocal diffusion systems with free boundaries}. arXiv:1907.04542v1.

\bibitem{1-DLin10}Y. H. Du and Z. G. Lin, {\it Spreading-Vanishing dichotomy in the diffusive logistic model with a free boundary}, SIAM J. Math. Anal., \textbf{42} (2010) 377-405.

\vspace{-1.5mm}\bibitem{DLdcds14} Y. H. Du and Z. G. Lin,
{\it The diffusive competition model with a free boundary: Invasion of a superior or inferior competitor}, Discrete Cont. Dyn. Syst. B, {\bf 19} (2014), 3105-3132.

\vspace{-1.5mm}\bibitem{6-GW15} 
J. S. Guo and C. H. Wu, {\it Dynamics for a two-species competition-diffusion model with two free boundaries}, Nonlinearity, \textbf{28} (2015), 1-27.

\vspace{-1.5mm}\bibitem{WZhang16} M. X. Wang and Y. Zhang,
{\it The time-periodic diffusive competition models with a free boundary and sign-changing growth rates}, Z. Angew. Math. Phys. {\bf 67}(5)(2016): 132. DOI: 10.1007/s00033-016-0729-9.

\vspace{-1.5mm}\bibitem{ZhaoW16} Y. G. Zhao and M. X. Wang, {\it Free boundary problems for the diffusive competition system in higher dimension with sign-changing coefficients}, IMA J. Appl. Math., 81(2016), 255-280.

\vspace{-1.5mm}\bibitem{7-WZna17} 
 M. X. Wang and Y. Zhang,
 {\it Note on a two-species competition-diffusion model with two free boundaries},
Nonlinear Anal.: TMA, \textbf{159} (2017), 458-467.

\vspace{-1.5mm}\bibitem{ZW2014} 
J. F. Zhao and M. X.  Wang,
\newblock {\it A free boundary problem of a predator-prey model with higher dimension and heterogeneous environment},
{Nonlinear Anal.: RWA}, 2014, \textbf{16}, 250-263.

\vspace{-1.5mm}\bibitem{Wcnsns15} M.X. Wang, {\it Spreading and vanishing in the diffusive prey-predator model with a free boundary},  Commun. Nonlinear Sci. Numer. Simulat., {\bf 23}(2015), 311-327.


\vspace{-1.5mm}\bibitem{ZWaa15}Y. Zhang and M. X. Wang,  {\it A free boundary problem of the ratio-dependent prey-predator model}, Applicable Anal., {\bf 94} (2015), 2147-2167.

\vspace{-1.5mm}\bibitem{2-W18}M. X. Wang, {\it Existence and uniqueness of solutions of free boundary problems in heterogeneous environments}. Discrete Cont. Dyn. Syst. B., \textbf{24}(2)(2019), 415-421.

\vspace{-1.5mm}\bibitem{10-hdu12} G. Bunting, Y. H. Du and K. Krakowski,
{\it Spreading speed revisited: Analysis of a free boundary model},
Networks and Heterogeneous Media, {\bf 42}(7) (2012), 583-603.

\vspace{-1.5mm}\bibitem{Nathan12}R. Natan, E. Klein, J. J. Robledo-Arnuncio and E. Revilla, {\it Dispersal kernels: Review}, in {\it Dispersal Ecology and Evolution}, J. Clobert, M. Baguette, T. G. Benton and J. M. Bullock, eds., Oxford University Press, Oxford, UK, 2012, 187-210.

\vspace{-1.5mm}\bibitem{7-BJFA16}H. Berestycki, J. Coville and H. Vo, {\it On the definition and the properties of the principal eigenvalue of some nonlocal operators}, J. Funct. Anal., \textbf{271}
(2016), 2701-2751.

\vspace{-1.5mm}\bibitem{8-BJMB16} H. Berestycki, J. Coville and H. Vo, {\it Persistence criteria for populations with non-local dispersion},
J. Math. Biol., \textbf{72} (2016), 1693-1745.

\vspace{-1.5mm}\bibitem{CDLL}J. F. Cao, Y. H. Du, F. Li and W. T. Li, {\it The dynamics of a
Fisher-KPP nonlocal diffusion model with free boundaries}, J. Functional Analysis, 2019. https://doi.org/10.1016/j.jfa.2019.02.013.

\vspace{-1.5mm}\bibitem{WWmma18}J. P. Wang and M. X. Wang, {\it The diffusive Beddington-DeAngelis predator-prey model with nonlinear prey-taxis and free boundary}, Math. Meth. Appl. Sci. \textbf{41} (2018), 6741-6762.

\vspace{-1.5mm}\bibitem{6-KLS10} C. Y. Kao, Y. Lou and W. X. Shen, {\it Random dispersal vs. non-local dispersal}, Discrete Contin. Dyn. Syst., \textbf{26} (2010), 551-596.

\vspace{-1.5mm}\bibitem{L-S-Y1968}O.A.~Ladyzenskaja, V.A.~Solonnikov and N. N.~Ural'ceva, {\it Linear and quasi-Linear equations of parabolic type}, Academic Press, New York, London, 1968.

\vspace{-1.5mm}\bibitem{WZ-dcdsa18} M. X. Wang and Q. Y. Zhang, \emph{Dynamics for the diffusive leslie-gower model with double free boundaries}, Discrete Cont. Dyn. Syst., {\bf 38}(5)(2018),  2591-2607.

\end{thebibliography}
\end{document}